\documentclass[a4paper,reqno]{amsart}
\usepackage{amsmath,amssymb,amsthm,graphicx,mathrsfs,url}
\usepackage[usenames,dvipsnames]{color}
\definecolor{darkred}{rgb}{0.4,0.1,0.1}
\definecolor{darkblue}{rgb}{0.1,0.1,0.4}

\usepackage{tikz}
\usetikzlibrary{datavisualization}
\usetikzlibrary{datavisualization.formats.functions}
\usetikzlibrary{hobby,backgrounds,patterns}
\usepackage[normalem]{ulem}

\numberwithin{equation}{section}
\theoremstyle{plain}

\newtheorem{theorem}{Theorem}[section]
\newtheorem{lemma}[theorem]{Lemma}

\newtheorem{proposition}[theorem]{Proposition}
\newtheorem{corollary}[theorem]{Corollary}

\theoremstyle{remark}
\newtheorem{remark}[theorem]{Remark}

\theoremstyle{definition}
\newtheorem{example}[theorem]{Example}

\newcommand\cH{\mathcal H}

\DeclareMathOperator{\diver}{div}

\DeclareMathOperator{\Real}{Re}

\definecolor{darkgreen}{rgb}{0.1,0.45,0.1}
\definecolor{darkblue}{rgb}{0.1,0.1,0.4}
\definecolor{darkgrey}{rgb}{0.5,0.5,0.5}

\definecolor{darkred}{rgb}{0.6,0.0,0.0}

\newcommand\void[1]{}

\renewcommand{\phi}{\varphi}

\def\sa{\mathfrak a}

      \def\cC{{\mathcal C}}
   \def\cE{{\mathcal E}}   
   \def\cH{{\mathcal H}}

\def\R{\mathbb{R}}
\def\C{\mathbb{C}}

\def\N{\mathbb{N}}

\renewcommand{\div}{\mathrm{div}\,}

\newcommand{\dom}{\mathrm{dom}\,}

\allowdisplaybreaks

\newcounter{counter_a}

\title[]{Inequalities between Neumann and Dirichlet Laplacian eigenvalues on planar domains}

\author[J.~Rohleder]{Jonathan Rohleder}
\address{Matematiska institutionen \\ Stockholms universitet \\
106 91 Stockholm \\
Sweden}
\email{jonathan.rohleder@math.su.se}

\begin{document}

\begin{abstract}
We generalize a classical inequality between the eigenvalues of the Laplacians with Neumann and Dirichlet boundary conditions on bounded, planar domains: in 1955, Payne proved that below the $k$-th eigenvalue of the Dirichlet Laplacian there exist at least $k + 2$ eigenvalues of the Neumann Laplacian, provided the domain is convex. It has, however, been conjectured that this should hold for any domain. Here we show that the statement indeed remains true for all simply connected planar Lipschitz domains. The proof relies on a novel variational principle.
\end{abstract}

\maketitle

\section{Introduction}

This paper is devoted to a classical question in spectral theory: given a bounded domain in Euclidean space, how many eigenvalues does the Neumann Laplacian have below the first, or the $k$-th, eigenvalue of the Dirichlet Laplacian? This question has recently been relevant, for instance, to the study of nodal domains \cite{CMS19} or the investigation of maxima and minima of eigenfunctions and the hot spots conjecture \cite{MPW22,S21}.

Let us assume that $\Omega \subset \R^d$, $d \geq 2$, is a bounded, connected Lipschitz domain, and let us denote by
\begin{align*}
 0 = \mu_1 < \mu_2 \leq \mu_3 \leq \dots
\end{align*}
the eigenvalues of the Laplacian $- \Delta_{\rm N}$ with Neumann boundary conditions and by
\begin{align*}
 0 < \lambda_1 < \lambda_2 \leq \lambda_3 \leq \dots
\end{align*}
the eigenvalues of the Laplacian $- \Delta_{\rm D}$ with Dirichlet boundary conditions, both counted according to their multiplicities. Classical variational principles imply immediately that $\mu_k \leq \lambda_k$ holds for all $k$. However, even
\begin{align}\label{eq:Filonov}
 \mu_{k + 1} < \lambda_k, \quad k \in \N,
\end{align}
is always true, which was proven by Filonov \cite{F05} in full generality, see also Friedlander's earlier paper \cite{F91} for a slightly weaker result and \cite{AM12} for a generalization of Friedlander's approach. It had in fact been known much earlier that $\mu_2 < \lambda_1$ holds for $d = 2$, see P\'olya \cite{P52}. For convex $\Omega$ even stronger inequalities are known: if $\Omega \subset \R^2$ is convex with $C^2$-smooth boundary, then the inequality
\begin{align*}
 \mu_{k + 2} < \lambda_k, \quad k \in \N,
\end{align*}
holds, proven by Payne \cite{P55}. Levine and Weinberger \cite{LW86} showed that for convex $\Omega$ whose boundary is $C^2$-smooth with H\"older continuous second derivatives, in space dimension $d$,
\begin{align}\label{eq:LW}
 \mu_{k+d} < \lambda_k, \quad k \in \N,
\end{align}
holds, and they proved further inequalities of the type $\mu_{k + R} < \lambda_k$ with $R < d$ for certain non-convex domains, with $R$ depending on curvature properties of the boundary; cf.\ also \cite{A86}. As pointed out in \cite{LW86}, by approximation it can be concluded from \eqref{eq:LW} that 
\begin{align}\label{eq:weakLW}
 \mu_{k + d} \leq \lambda_k, \quad k \in \N,
\end{align}
is true for any convex bounded domain $\Omega \subset \R^d$. The index shift cannot be improved further; for instance, on a disk in the plane, one has $\mu_4 > \lambda_1$; cf.\ Example \ref{ex:disk} below.

It has been conjectured that the inequality \eqref{eq:weakLW} should be true for arbitrary bounded domains in $\R^d$, see, e.g., \cite[Conjecture 3.2.42]{LMP}. However, the proofs by Payne and Levine and Weinberger make explicit use of the signs of curvatures of the boundary and do not extend to non-convex domains. A suspected counterexample on a non-convex domain proposed in \cite[p.\ 207]{LW86} based on numerical computations turned out not to hold; cf.\ \cite[Remark 3.9]{BLP09}. On the other hand, a possible path towards proving \eqref{eq:weakLW} for more general domains was proposed in \cite{BLP09}, but without actually proving it; see Remark~3.9 there. Very recently, the question and its significance have been pointed out again in \cite{CMS19}.

For domains in the plane, the present article is making the step from convex to simply connected: we prove that for any simply connected, bounded Lipschitz domain $\Omega \subset \R^2$ the inequality \eqref{eq:weakLW} holds, that is,
\begin{align}\label{eq:dasGrosseZiel}
 \mu_{k + 2} \leq \lambda_k, \quad k \in \N,
\end{align}
is true. Furthermore, we provide sufficient conditions for this inequality to be strict, see Theorem~\ref{thm:main} below. In particular, we show that
\begin{align*}
 \mu_3 < \lambda_1
\end{align*}
holds on each simply connected planar Lipschitz domain. The latter can for instance be used to exclude closed nodal lines for any third eigenfunction of the Neumann Laplacian on a simply connected domain in the plane; cf.\ Corollary \ref{cor:firstDirichlet}.  

Our method of proving \eqref{eq:dasGrosseZiel} is variational, but differs substantially from all earlier attempts: instead of using the classical variational principles for Neumann and Dirichlet Laplacian eigenvalues, we derive and use a new variational principle, which incorporates the eigenvalues of both operators simultaneously, and whose minimizers are gradients of eigenfunctions instead of the eigenfunctions themselves. In fact, denoting by $\eta_1 \leq \eta_2 \leq \dots$ the union of the positive eigenvalues of $- \Delta_{\rm N}$ and $- \Delta_{\rm D}$, including multiplicities, we show that on any simply connected domain $\Omega \subset \R^2$,
\begin{align*}
 \eta_k = \min_{\substack{U \subset \cH_\sa \\ \dim U = k}} \max_{\substack{u \in U \\ u \neq 0}} \frac{\int_\Omega \left( |\diver u|^2 + |\omega (u)|^2 \right)}{\int_\Omega |u|^2}
\end{align*}
holds for all $k \in \N$, where $\omega (u) = \partial_1 u_2 - \partial_2 u_1$ is the vorticity of a vector field $u$ and $\cH_{\sa}$ consists of all vector fields $u \in L^2 (\Omega)^2$ such that $\diver u, \omega (u) \in L^2 (\Omega)$ and $u |_{\partial \Omega}$ is tangential; see Theorem \ref{thm:sameEV}. After establishing this, we may essentially use vector fields $u$ whose components are eigenfunctions of $- \Delta_{\rm D}$ as test functions to obtain \eqref{eq:dasGrosseZiel}. A similar variational principle was derived recently by the author in connection with the hot spots conjecture \cite{Rprep}. These joint variational principles for Neumann and Dirichlet eigenvalues seem to be unique to the case $d = 2$ and do not have any obvious analogs in higher space dimensions. Therefore it is likely that the approach taken here cannot be extended to higher dimensions.

Finally, we would like to mention that eigenvalue inequalities of the types \eqref{eq:Filonov} and \eqref{eq:LW} also were studied in other situations. These include Robin \cite{GM09,L88,R14} and mixed \cite{LR17} boundary conditions, Schr\"odinger operators with potentials \cite{BRS18,R21}, the Laplacian on the Heisenberg group \cite{FL10,H08} or on manifolds \cite{AL97,M91}, the Stokes operator \cite{DEprep,K10} and polyharmonic operators \cite{Lprep,P19}.

\section{Preliminaries}\label{sec:prel}

Let $\Omega \subset \R^2$ be a bounded, connected Lipschitz domain; cf., e.g., \cite[Chapter 3]{McL}. By Rademacher's theorem, the unit normal vector $\nu (x) = (\nu_1 (x), \nu_2 (x))^\top$ and the unit tangent vector $\tau (x) = (\tau_1 (x), \tau_2 (x))^\top = (- \nu_2 (x), \nu_1 (x))^\top$ are uniquely defined for almost all $x \in \partial \Omega$. We denote by $H^1 (\Omega)$ and $H^2 (\Omega)$ the usual Sobolev spaces on $\Omega$ of order one and two, respectively. On the boundary $\partial \Omega$ we will make use of the Sobolev space $H^{1/2} (\partial \Omega)$ of order $1/2$ and its dual space $H^{- 1/2} (\partial \Omega)$. We write $(\cdot, \cdot)_{\partial \Omega}$ for the sesquilinear duality between $H^{- 1/2} (\partial \Omega)$ and $H^{1/2} (\partial \Omega)$. Recall that the trace map
\begin{align*}
 C^\infty (\overline{\Omega}) \ni \varphi \mapsto \varphi |_{\partial \Omega}
\end{align*}
extends uniquely to a bounded, everywhere defined, surjective operator $H^1 (\Omega) \to H^{1/2} (\partial \Omega)$; for $\phi \in H^1 (\Omega)$ we again write $\phi |_{\partial \Omega}$ for its trace.

In order to define a further boundary map, we consider the space
\begin{align*}
 \cE (\Omega) := \left\{ u = \binom{u_1}{u_2} \in L^2 (\Omega)^2 : \diver u \in L^2 (\Omega) \right\}
\end{align*}
of square-integrable vector fields with square-integrable divergence (taken in the sense of distributions). Equipped with the norm defined by
\begin{align*}
 \|u\|_{\cE (\Omega)}^2 = \int_\Omega \left( |u_1|^2 + |u_2|^2 + |\diver u|^2 \right),
\end{align*}
$\cE (\Omega)$ is a Hilbert space. The mapping
\begin{align*}
 C^\infty (\overline \Omega)^2 \ni u \mapsto \langle u |_{\partial \Omega}, \nu \rangle 
\end{align*}
extends by continuity to a bounded operator $\cE (\Omega) \to H^{- 1/2} (\partial \Omega)$, see \cite[Chapter XIX, \S 1, Theorem 2]{DL}, and we write $\langle u |_{\partial \Omega}, \nu \rangle$ for the image of $u \in \cE (\Omega)$ under this mapping, called the {\em normal trace} of $u$. In particular, the integration-by-parts formula
\begin{align}\label{eq:PI}
 \int_\Omega \langle u, \nabla \phi\rangle + \int_\Omega  \overline{\phi} \diver u = \big( \langle u |_{\partial \Omega}, \nu\rangle, \phi \big)_{\partial \Omega}
\end{align}
holds for all $\phi \in H^1 (\Omega)$ and all $u \in \cE (\Omega)$. For later use we note an analogous formula: define
\begin{align*}
 \nabla^\perp \phi := \binom{- \partial_2 \phi}{\partial_1 \phi} \quad \text{and} \quad \omega (u) := \partial_1 u_2 - \partial_2 u_1.
\end{align*}
The following lemma is an immediate consequence of \eqref{eq:PI}.

\begin{lemma}
Assume that $\Omega$ is a bounded, connected Lipschitz domain and let $u = (u_1, u_2)^\top \in L^2 (\Omega)^2$ such that $\omega (u) \in L^2 (\Omega)$. Then $u^\perp := (- u_2, u_1)^\top$ belongs to $\cE (\Omega)$, and
\begin{align}\label{eq:PIomega}
 \int_\Omega \langle u, \nabla^\perp \phi\rangle + \int_\Omega  \overline{\phi} \, \omega (u) = - \big( \langle u^\perp |_{\partial \Omega}, \nu\rangle, \phi \big)_{\partial \Omega}
\end{align}
holds for all $\phi \in H^1 (\Omega)$.
\end{lemma}

The normal trace on $\cE (\Omega)$ can be used to define a weak version of the normal derivative: let $\psi \in H^1 (\Omega)$ such that $\Delta \psi$, taken distributionally, belongs to $L^2 (\Omega)$. Then $u = \nabla \psi \in \cE (\Omega)$ and, hence, the {\em normal derivative}
\begin{align*}
 \partial_\nu \psi |_{\partial \Omega} := \langle \nabla \psi |_{\partial \Omega}, \nu \rangle \in H^{- 1/2} (\partial \Omega)
\end{align*}
is well-defined by the above. For $u \in H^2 (\Omega)$, $\partial_\nu \psi |_{\partial \Omega}$ can alternatively be defined by taking the trace $\nabla \psi |_{\partial \Omega}\in H^{1/2} (\partial \Omega)^2$ (though even in this case $\partial_\nu \psi$ does not necessarily belong to $H^{1/2} (\partial \Omega)$ as $\nu$ in general is merely bounded).

An important part of our analysis will be based on the well-known structure of the space of vector fields $L^2 (\Omega)^2$. We define
\begin{align*}
 H := \left\{ u \in L^2 (\Omega)^2 : \diver u = 0~\text{in}~\Omega, \left\langle u |_{\partial \Omega}, \nu \right\rangle = 0 \right\}.
\end{align*}
Then $H$ is a closed subspace of $L^2 (\Omega)^2$ and the  Helmholtz decomposition
\begin{align}\label{eq:Helmholtz}
 L^2 (\Omega)^2 = \nabla H^1 (\Omega) \oplus H
\end{align}
holds, see, e.g., \cite[Chapter XIX, \S 1, Theorem 4]{DL}. Further decomposition of $H$ yields
\begin{align*}
 H = \nabla^\perp H_0^1 (\Omega) \oplus H_{\rm c},
\end{align*}
for the space
\begin{align*}
 H_{\rm c} := \left\{u \in H : \omega (u) = 0~\text{in}~\Omega \right\}
\end{align*}
of curl-free vector fields in $H$. In particular, if $\Omega$ is simply connected, then the space $H_{\rm c}$ is trivial and the Helmholtz decomposition \eqref{eq:Helmholtz} takes the form
\begin{align*}
 L^2 (\Omega)^2 = \nabla H^1 (\Omega) \oplus \nabla^\perp H_0^1 (\Omega),
\end{align*}
see, e.g., \cite[Lemma 2.10]{K10}.

Next, we define the Laplacians $- \Delta_{\rm N}$ and $- \Delta_{\rm D}$ on $\Omega$ with Neumann and Dirichlet boundary conditions as
\begin{align*}
 - \Delta_{\rm N} u & = - \Delta u, \quad \dom (- \Delta_{\rm N}) = \left\{u \in H^1 (\Omega) : \Delta u \in L^2 (\Omega), \partial_\nu u |_{\partial \Omega} = 0 \right\},
\end{align*}
and
\begin{align*}
 - \Delta_{\rm D} u & = - \Delta u, \quad \dom (- \Delta_{\rm D}) = \left\{u \in H^1 (\Omega) : \Delta u \in L^2 (\Omega), u |_{\partial \Omega} = 0 \right\}.
\end{align*}
Both are unbounded, self-adjoint operators in $L^2 (\Omega)$, and their spectra consist of isolated eigenvalues of finite multiplicities. Let
\begin{align*}
 0 = \mu_1 < \mu_2 \leq \mu_3 \leq \dots 
\end{align*}
be an enumeration of the eigenvalues of $- \Delta_{\rm N}$, counted with multiplicities, and let
\begin{align*}
 \lambda_1 < \lambda_2 \leq \lambda_3 \leq \dots 
\end{align*}
be the eigenvalues of $- \Delta_{\rm D}$, also these counted according to their multiplicities. As $\Omega$ is connected, the lowest eigenvalue $\mu_1$ of $- \Delta_{\rm N}$ is zero, with multiplicity one and corresponding eigenspace given by the constant functions. The lowest eigenvalue of $- \Delta_{\rm D}$ is positive and has multiplicity one as well.

We conclude this section with a useful little lemma. Informally, its meaning is that for a function being constant on $\partial \Omega$, the derivative in tangential direction must vanish on the whole boundary. We provide a short proof, especially to point out that the statement does not require any additional regularity of the boundary.

\begin{lemma}\label{lem:DTangential}
Assume that $\Omega$ is a bounded, connected Lipschitz domain. Then for each $\phi \in \dom (- \Delta_{\rm D})$ the vector field $\nabla^\perp \phi$ belongs to $\cE (\Omega)$, and
\begin{align*}
 \left\langle \nabla^\perp \phi |_{\partial \Omega}, \nu \right\rangle = 0
\end{align*}
holds.
\end{lemma}

\begin{proof}
As $\phi \in \dom (- \Delta_{\rm D})$, we have $\nabla^\perp \phi \in L^2 (\Omega)^2$ and $\diver \nabla^\perp \phi = 0$. Thus $\nabla^\perp \phi \in \cE (\Omega)$. Since $\dom (- \Delta_{\rm D}) \subset H_0^1 (\Omega)$, there exists a sequence $(\phi_n) \subset C_0^\infty (\Omega)$ such that $\phi_n \to \phi$ in $H^1 (\Omega)$. Then $\nabla^\perp \phi_n \to \nabla^\perp \phi$ in $\cE (\Omega)$, as $\diver \nabla^\perp \phi_n = 0$ for all $n$. Furthermore,
\begin{align*}
 \left\langle \nabla^\perp \phi |_{\partial \Omega}, \nu \right\rangle & = \lim_{n \to \infty} \left\langle \nabla^\perp \phi_n |_{\partial \Omega}, \nu \right\rangle = 0,
\end{align*}
by the boundedness of the normal trace on $\cE (\Omega)$, giving the statement of the lemma.
\end{proof}

\section{A novel variational principle for Neumann and Dirichlet Laplacian eigenvalues}\label{sec:A}

In this section we define a self-adjoint operator $A$ in the space $L^2 (\Omega)^2$ whose spectrum, in the case of a simply connected domain, equals the union of the positive eigenvalues of the Neumann and Dirichlet Laplacians, including multiplicities. This leads, in particular, to a novel variational principle for Neumann and Dirichlet Laplacian eigenvalues. A related variational principle was developed recently by the author in connection with the hot spots conjecture \cite{Rprep}; cf.\ Remark \ref{rem:hotSpots}.

First of all, we only assume that $\Omega \subset \R^2$ is a bounded, connected Lipschitz domain. We define a sesquilinear form $\sa$ in $L^2 (\Omega)^2$ via
\begin{align*}
 \sa [u, v] & = \int_\Omega \left( \diver u \, \overline{\diver v} + \omega (u) \overline{\omega (v)} \right),
\end{align*}
for $u$ and $v$ in its domain
\begin{align*}
 \dom \sa & = \left\{ u \in L^2 (\Omega)^2 : \diver u, \omega (u) \in L^2 (\Omega), \langle u |_{\partial \Omega}, \nu \rangle = 0 \right\}.
\end{align*}
We point out that the normal trace $\langle u |_{\partial \Omega}, \nu\rangle$ is well-defined in the sense of Section~\ref{sec:prel} for $u \in \dom \sa$, since $\dom \sa \subset \cE (\Omega)$. 

In the following, we make use of the theory of semi-bounded sesquilinear forms in Hilbert spaces and associated self-adjoint operators as to be found in \cite[Chapter~VI]{Kato} or \cite[Section X.3]{RS75}. We denote by $(\cdot, \cdot)$ the inner product in $L^2 (\Omega)^2$ and by $\|\cdot\|$ the corresponding norm. The sesqulinear form $\sa$ is symmetric and non-negative definite; in particular,
\begin{align}\label{eq:aInnerProd}
 (u, v)_\sa := (u, v) + \sa [u, v]
\end{align}
defines an inner product on the space $\cH_\sa := \dom \sa$.

\begin{proposition}\label{prop:defA}
Let $\Omega \subset \R^2$ be a bounded, connected Lipschitz domain. The sesquilinear form $\sa$ has a dense domain in $L^2 (\Omega)^2$ and is closed, that is, $\cH_\sa$, equipped with the inner product \eqref{eq:aInnerProd}, is a Hilbert space. In particular, there exists a self-adjoint operator $A$ in $L^2 (\Omega)^2$ such that
\begin{align*}
 \sa [u, v] = (A u, v)
\end{align*}
holds for all $u \in \dom A \subset \dom \sa$ and $v \in \dom \sa$. Moreover, a vector field $u \in \dom \sa$ belongs to $\dom A$ if and only if there exists a vector field $w \in L^2 (\Omega)^2$ such that 
\begin{align*}
 \sa [u, v] = (w, v)
\end{align*}
holds for all $v \in \dom \sa$; in this case, $A u = w$.
\end{proposition}

\begin{proof}
We only need to prove the mentioned properties of $\sa$. The characterization of $A$ and its domain follows from abstract theory, see, e.g., \cite[Section X.3]{RS75}. As $C_0^\infty (\Omega)^2 \subset \dom \sa$, $\dom \sa$ is dense in $L^2 (\Omega)^2$. Therefore it remains to show that $\sa$ is closed. To this end, let $(u^n)$ be a Cauchy sequence in $\cH_\sa$, that is, 
\begin{align*}
 \int_\Omega \left( |\diver u^n - \diver u^m|^2 + |\omega (u^n) - \omega (u^m)|^2 \right) + \|u^n - u^m\|^2 \to 0
\end{align*}
as $n, m \to \infty$. By the completeness of $L^2 (\Omega)$, there exist $f, g \in L^2 (\Omega), u \in L^2 (\Omega)^2$, such that 
\begin{align*}
 \diver u^n \to f, \quad \omega (u^n) \to g, \quad \text{and} \quad u^n \to u
\end{align*}
in $L^2 (\Omega)$, respectively $L^2 (\Omega)^2$. For all $\phi \in C_0^\infty (\Omega)$ it follows with the help of \eqref{eq:PI}
\begin{align*}
 \int_\Omega f \overline{\phi} = \lim_{n \to \infty} \int_\Omega \diver u^n \overline{\phi} = - \lim_{n \to \infty} \int_\Omega \langle u^n, \nabla \phi \rangle = - \int_\Omega \langle u, \nabla \phi\rangle = \int_\Omega \diver u \overline{\phi},
\end{align*}
which implies $\diver u = f \in L^2 (\Omega)$. Similarly, by \eqref{eq:PIomega},
\begin{align*}
 \int_\Omega g \overline{\phi} = \lim_{n \to \infty} \int_\Omega \omega (u^n) \overline{\phi} = - \lim_{n \to \infty} \int_\Omega \langle u^n, \nabla^\perp \phi \rangle = - \int_\Omega \langle u, \nabla^\perp \phi\rangle = \int_\Omega \omega(u) \overline{\phi},
\end{align*}
giving $\omega (u) = g \in L^2 (\Omega)$. Finally, for the boundary condition, for all $\psi \in H^1 (\Omega)$, we have
\begin{align*}
 \int_\Omega \langle u, \nabla \psi\rangle + \int_\Omega \overline{\psi} \diver u = \lim_{n \to \infty} \left( \int_\Omega \langle u^n, \nabla \psi\rangle + \int_\Omega \overline{\psi} \diver u^n \right) = 0,
\end{align*}
that is, $\langle u |_{\partial \Omega}, \nu \rangle = 0$ by \eqref{eq:PI}. Consequently, $u \in \cH_\sa$ and $\|u^n - u\|_\sa \to 0$ as $n \to \infty$. Hence, $\sa$ is closed.
\end{proof}

Next we compute the spectrum of $A$. Recall that the Helmholtz decomposition reads
\begin{align}\label{eq:HelmholtzRecall}
 L^2 (\Omega)^2 = \nabla H^1 (\Omega) \oplus \nabla^\perp H_0^1 (\Omega) \oplus H_{\rm c}
\end{align}
and that
\begin{align*}
 H_{\rm c} := \left\{u \in L^2 (\Omega)^2 : \diver u = \omega (u) = 0~\text{in}~\Omega, \langle u |_{\partial \Omega}, \nu \rangle = 0 \right\}
\end{align*}
is trivial in case $\Omega$ is simply connected; cf.\ Section \ref{sec:prel}. In the following proposition, we choose orthonormal bases $(\psi_1, \psi_2, \dots)$ and $(\phi_1, \phi_2, \dots)$ of $L^2 (\Omega)$ such that $- \Delta_{\rm N} \psi_j = \mu_j \psi_j$ and $- \Delta_{\rm D} \phi_j = \lambda_j \phi_j$ hold for all $j \in \N$.

\begin{proposition}\label{prop:spec}
Let $\Omega \subset \R^2$ be a bounded, connected Lipschitz domain. Then the following hold.
\begin{enumerate}
 \item For each $j \in \N$, $j \geq 2$, the vector field $\nabla \psi_j$ is non-trivial and belongs to $\ker (A - \mu_j)$. Moreover, the fields $\frac{1}{\sqrt{\mu_j}} \nabla \psi_j$, $j = 2, 3, \dots$, form an orthonormal basis of $\nabla H^1 (\Omega)$.
 \item For each $j \in \N$, the vector field $\nabla^\perp \phi_j$ is non-trivial and belongs to $\ker (A - \lambda_j)$. Moreover, the fields $\frac{1}{\sqrt{\lambda_j}} \nabla^\perp \phi_j$, $j = 1, 2, \dots$, form an orthonormal basis of $\nabla^\perp H_0^1 (\Omega)$.
 \item $\ker A = H_{\rm c}$.
\end{enumerate}
In particular, the spectrum of $A$, taking into account multiplicities, consists of $\mu_2, \mu_3, \dots$, $\lambda_1, \lambda_2, \dots$, and an eigenvalue zero of multiplicity $\dim H_{\rm c}$. If $\Omega$ is simply connected, then the spectrum of $A$ equals the union of the positive spectra of the Neumann and Dirichlet Laplacians, counted with multiplicities. 
\end{proposition}

\begin{proof}
To show (i), let $j \geq 2$. Clearly $\nabla \psi_j$ is non-trivial and belongs to $\dom \sa$ as $\diver \nabla \psi_j = \Delta \psi_j = - \mu_j \psi_j \in L^2 (\Omega)$, $\omega (\nabla \psi_j) = 0$, and
\begin{align*}
 \langle u |_{\partial \Omega}, \nu \rangle = \langle \nabla \psi |_{\partial \Omega}, \nu\rangle = \partial_\nu \psi |_{\partial \Omega} = 0
\end{align*}
due to the Neumann boundary condition. Thus for all $v \in \dom \sa$,
\begin{align*}
 \sa [\nabla \psi_j, v] & = \int_\Omega (\Delta \psi_j) \, \overline{\diver v} = - \mu_j \int_\Omega \psi_j \overline{\diver v} = \mu_j \int_\Omega \langle \nabla \psi_j, v \rangle,
\end{align*}
where we have used $\langle v |_{\partial \Omega}, \nu\rangle = 0$. Hence $\nabla \psi_j \in \dom A$ and $A \nabla \psi_j = \mu_j \nabla \psi_j$. Moreover, the vector fields $\frac{1}{\sqrt{\mu_j}} \nabla \psi_j$, $j = 2, 3, \dots$, are orthonormal, since
\begin{align*}
 \int_\Omega \left\langle \frac{1}{\sqrt{\mu_j}} \nabla \psi_j, \frac{1}{\sqrt{\mu_k}} \nabla \psi_k \right\rangle & = - \frac{1}{\sqrt{\mu_j \mu_k}} \int_\Omega (\Delta \psi_j) \overline{\psi_k} = \frac{\sqrt{\mu_j}}{\sqrt{\mu_k}} \int_\Omega \psi_j \overline{\psi_k}
\end{align*}
and the $\psi_j$ are orthonormal in $L^2 (\Omega)$. Finally, let $\psi \in H^1 (\Omega)$ such that $\nabla \psi$ is orthogonal in $L^2 (\Omega)^2$ to $\nabla \psi_j$ for each $j \geq 2$. Then
\begin{align*}
 0 & = \int_\Omega \langle \nabla \psi_j, \nabla \psi \rangle = \mu_j \int_\Omega \psi_j \overline{\psi}, \quad j \geq 2.
\end{align*}
As the $\psi_j$, $j = 1, 2, \dots$ form an orthonormal basis of $L^2 (\Omega)$ and $\psi_1$ is constant, this implies that $\psi$ is constant, i.e.\ $\nabla \psi = 0$. Hence $\frac{1}{\sqrt{\mu_j}} \nabla \psi_j$, $j \geq 2$, form an orthonormal basis of $\nabla H^1 (\Omega)$.

The proof of (ii) is completely analogous and most details will be left to the reader. For instance, for each $j \in \N$, $\nabla^\perp \phi_j$ belongs to $\dom \sa$ as $\diver \nabla^\perp \phi_j = 0$, $\omega (\nabla^\perp \phi_j) = \Delta \phi_j = - \lambda_j \phi_j \in L^2 (\Omega)$, and $\langle \nabla^\perp \phi_j |_{\partial \Omega}, \nu \rangle = 0$ by Lemma \ref{lem:DTangential}. Moreover, for each $v \in \dom \sa$,
\begin{align*}
 \sa [\nabla^\perp \phi_j, v] & = \int_\Omega (\Delta \phi_j) \overline{\omega(v)} = - \lambda_j \int_\Omega \phi_j \overline{\omega (v)} = \lambda_j \int_\Omega \langle \nabla^\perp \phi_j, v\rangle
\end{align*}
by \eqref{eq:PIomega}. Thus $\nabla^\perp \phi_j \in \ker (A - \lambda_j)$.

To show (iii), let $w \in H_c \subset \dom \sa$ and $v \in \dom \sa$. Then $\diver w = \omega (w) = 0$ in $\Omega$ and, thus,
\begin{align*}
 \sa [w, v] = \int_\Omega \left( \diver w \, \overline{\diver v} + \omega (w) \overline{\omega (v)} \right) = 0.
\end{align*}
Hence, $w \in \dom A$ and $A w = 0$. We have shown $H_{\rm c} \subset \ker A$. However, as $A$ is a self-adjoint operator in $L^2 (\Omega)^2$ and, by the Helmholtz decomposition~\eqref{eq:HelmholtzRecall}, $H_{\rm c}$ together with the eigenfunctions found in parts (i) and (ii) of this proposition span the whole space, no further elements can exist in the kernel. In particular, we have determined the whole spectrum of $A$ and, hence, the remaining assertions of the proposition follow immediately.
\end{proof}

\begin{example}
Consider the square $\Omega = (0, \pi)^2$. Since $\Omega$ is simply connected, $A$ has only strictly positive eigenvalues, which we can compute explicitly. By separation of variables, the first eigenvalues of $- \Delta_{\rm N}$ are $0, 1, 1, 2, 4, 4, 5, 5, 8$ and the first eigenvalues of $- \Delta_{\rm D}$ are $2, 5, 5, 8$, including multiplicities. Thus the eigenvalues of $A$ are
\begin{align*}
 1, 1, 2, 2, 4, 4, 5, 5, 5, 5, 8, 8, \dots
\end{align*}
\end{example}

The statement of Proposition \ref{prop:spec} on the spectrum of $A$ implies immediately the following variational principle; cf.\ \cite[Section XIII.1]{RS78}. Recall that
\begin{align*}
 \cH_\sa = \dom \sa & = \left\{ u \in L^2 (\Omega)^2 : \diver u, \omega (u) \in L^2 (\Omega), \langle u |_{\partial \Omega}, \nu\rangle = 0\right\}
\end{align*}
and 
\begin{align*}
 \sa [u] := \sa [u, u] = \int_\Omega \left( |\diver u|^2 + |\omega (u)|^2 \right), \quad u \in \dom \sa.
\end{align*}

\begin{theorem}\label{thm:sameEV}
Let $\Omega$ be a bounded, connected Lipschitz domain and let $\eta_1 \leq \eta_2 \leq \dots$ denote the positive eigenvalues of $A$, i.e.\ the union of the positive eigenvalues of the Neumann and Dirichlet Laplacians, counted with multiplicities. Then
\begin{align*}
 \eta_k = \min_{\substack{U \subset \cH_\sa \\ U \perp H_{\rm c} \\ \dim U = k}} \max_{\substack{u \in U \\ u \neq 0}} \frac{\int_\Omega \left( |\diver u|^2 + |\omega (u)|^2 \right)}{\int_\Omega |u|^2}
\end{align*}
holds for all $k \in \N$. In particular, if $\Omega$ is simply connected, then
\begin{align}\label{eq:minMaxSC}
 \eta_k = \min_{\substack{U \subset \cH_\sa \\ \dim U = k}} \max_{\substack{u \in U \\ u \neq 0}} \frac{\int_\Omega \left( |\diver u|^2 + |\omega (u)|^2 \right)}{\int_\Omega |u|^2}
\end{align}
holds for all $k \in \N$.
\end{theorem} 

We conclude this section with a few remarks which are not essential for the main result of this article, but may be useful for the understanding of the operator $A$.

\begin{remark}
The operator $A$ has been defined above in a weak sense, using the sesquilinear form $\sa$. This is sufficient for the purposes of this article. However, the action and domain of $A$ can be computed using integration by parts. For instance, for $u \in \dom A$ and $v \in C_0^\infty (\Omega)^2 \subset \dom \sa$ we have
\begin{align*}
 (A u, v) = \sa [u, v] = - \int_\Omega \left\langle \nabla \div u + \nabla^\perp \omega (u), v \right\rangle = - \int_\Omega \langle \Delta u, v\rangle,
\end{align*}
where the derivatives are first taken in the sense of distributions; the obtained identity then shows
\begin{align*}
 - \Delta u = A u \in L^2 (\Omega)^2,
\end{align*}
that is, $A$ acts as negative Laplacian. Performing a similar integration by parts for arbitrary $v \in \dom \sa$ and investigating the effecting boundary terms yields that $\dom A$ consists of all vector fields $u \in \dom \sa$ such that $\Delta u \in L^2 (\Omega)^2$ and $\omega (u) |_{\partial \Omega} = 0$ in an appropriate weak sense.

Moreover, it follows from Proposition \ref{prop:spec} that $A$ decomposes orthogonally with respect to the Helmholtz decomposition~\eqref{eq:HelmholtzRecall}, 
\begin{align*}
 A = A_{\rm N} \oplus A_{\rm D} \oplus A_{\rm c},
\end{align*}
where $A_{\rm N}$, $A_{\rm D}$ and $A_{\rm c}$ are self-adjoint operators in $\nabla H^1 (\Omega), \nabla^\perp H_0^1 (\Omega)$ and $H_{\rm c}$, respectively. In particular, the spectrum of $A_{\rm N}$ equals the positive spectrum of $- \Delta_{\rm N}$, the spectrum of $A_{\rm D}$ equals the spectrum of $- \Delta_{\rm D}$, and $A_{\rm c}$ is the zero operator on $H_{\rm c}$.
\end{remark}

\begin{remark}\label{rem:hotSpots}
In \cite{Rprep} the author of this article established a variational principle similar to \eqref{eq:minMaxSC}, for the numbers $\eta_1 \leq \eta_2 \leq \dots$ being the union of the positive eigenvalues of $- \Delta_{\rm N}$ and $- \Delta_{\rm D}$: for simply connected $\Omega$ with piecewise smooth boundary whose corners, if any, are convex, he showed 
\begin{align}\label{eq:minMaxHotSpots}
 \eta_k = \min_{\substack{U \subset \cH_{\rm N} \\ \dim U = k}} \max_{u = \binom{u_1}{u_2} \in U \setminus \{0\}} \frac{\int_\Omega \left( |\nabla u_1|^2 + |\nabla u_2|^2 \right) - \int_{\partial \Omega} \kappa |u|^2}{\int_\Omega |u|^2},
\end{align}
where $\kappa$ is the signed curvature of $\partial \Omega$, defined everywhere except at corners, and
\begin{align*}
 \cH_{\rm N} := \left\{ u \in H^1 (\Omega)^2 : \left\langle u |_{\partial \Omega}, \nu \right\rangle = 0 \right\}.
\end{align*}
In fact, under those regularity assumptions, elliptic regularity implies that $\cH_{\sa} = \cH_{\rm N}$, and a computation similar to the one in the proof of \cite[Lemma 3.3]{Rprep} shows that the right-hand sides of \eqref{eq:minMaxSC} and \eqref{eq:minMaxHotSpots} coincide. 
\end{remark}

\begin{remark}
If $\partial \Omega$ is, e.g., piecewise $C^{1,1}$-regular and possible corners are convex, then $\cH_{\sa} \subset H^1 (\Omega)^2$; cf.\ the previous remark. In this case, $\cH_\sa$ is compactly embedded into $L^2 (\Omega)^2$ and, hence, $A$ has a purely discrete spectrum, implying that $\cH_{\rm c} = \ker A$ is finite-dimensional. However, it can in fact be shown that $\dim H_{\rm c}$ equals the first Betti number of $\Omega$, see, e.g.,~\cite[Chapter 3]{MMMT16}.
\end{remark}

\begin{remark}
As an alternative to the variational principle obtained in Theorem~\ref{thm:sameEV} above, one could restrict the sesquilinear form $\sa$ to the subspaces $\nabla H^1 (\Omega)$ and $\nabla^\perp H_0^1 (\Omega)$, respectively, to obtain variational principles for the positive eigenvalues of $- \Delta_{\rm N}$ and $- \Delta_{\rm D}$ separately. For instance, the form
\begin{align*}
 \sa_{\rm N} := \sa \upharpoonright (\dom \sa \cap \nabla H^1 (\Omega)) \quad \text{in}~\nabla H^1 (\Omega)
\end{align*}
gives rise to the min-max principle
\begin{align*}
 \mu_{k + 1} = \min_{\substack{U \subset \cH_\sa \cap \nabla H^1 (\Omega) \\ \dim U = k}} \max_{\substack{u \in U \\ u \neq 0}} \frac{\int_\Omega |\diver u|^2}{\int_\Omega |u|^2}
\end{align*}
for all $k \in \N$. However, this principle is not sufficient for the argument carried out in the proof of the main result of this article below, as the test functions used there are not gradient fields, i.e., they do not lie in the required test space.
\end{remark}

\begin{remark}
The operator $A$ constructed in this section has a natural interpretation in the language of differential forms, which the author was not aware of initially. After a preprint version of the present manuscript had come out, this interpretation was worked out independently by Fries, Goffeng and Miranda \cite{FGM24+}, Hua, M\"unch and Zhang \cite{HMZ24+}, and Muravyev \cite{M24+}; the manuscript \cite{HMZ24+} also provides a generalization to a class of surfaces with boundary. In fact, upon identifying classical vector fields with 1-forms, the identity \eqref{eq:Helmholtz} corresponds to the Hodge decomposition, and the operator $A$ is the Hodge Laplacian with absolute boundary conditions. We refer to the mentioned manuscripts for more details.
\end{remark}

\section{Dirichlet-Neumann eigenvalue inequalities on simply connected domains}

The following theorem is the main result of this article.

\begin{theorem}\label{thm:main}
Let $\Omega \subset \R^2$ be a bounded, simply connected Lipschitz domain. Then 
\begin{align}\label{eq:almostGoal}
 \mu_{k + 2} \leq \lambda_k
\end{align}
holds for all $k \in \N$. If $\lambda_k$ is a simple eigenvalue of $- \Delta_{\rm D}$, or if $\partial \Omega$ contains a straight line segment (i.e.\ a non-empty, relatively open set $\Sigma \subset \partial \Omega$ exists on which the normal vector $\nu$ is constant), then
\begin{align*}
 \mu_{k + 2} < \lambda_k
\end{align*}
holds.
\end{theorem}

We emphasize the following special case of Theorem \ref{thm:main}.

\begin{corollary}
Assume that $\Omega \subset \R^2$ is a bounded, simply connected Lipschitz domain whose boundary is a polygon. Then
\begin{align}\label{eq:wasWillManMehr}
 \mu_{k + 2} < \lambda_k
\end{align}
holds for all $k \in \N$.
\end{corollary}

As stressed for instance in \cite{CMS19}, the number of Neumann Laplacian eigenvalues strictly below $\lambda_1$ is of special interest, amongst others for the study of nodal domains. Since $\lambda_1$ is always a simple eigenvalue, we obtain the following corollary.

\begin{corollary}\label{cor:firstDirichlet}
Let $\Omega \subset \R^2$ be a bounded, simply connected Lipschitz domain. Then
\begin{align}\label{eq:strictFirst}
 \mu_3 < \lambda_1.
\end{align}
In particular, none of the eigenfunctions of $- \Delta_{\rm N}$ corresponding to $\mu_2$ or $\mu_3$ has a closed nodal line.
\end{corollary}

\begin{proof}
The inequality \eqref{eq:strictFirst} follows from Theorem \ref{thm:main} and the fact that the lowest eigenvalue of $- \Delta_{\rm D}$ is always simple. Furthermore, closed nodal lines can be excluded by the following standard argument: let $\psi$ be an eigenfunction of $- \Delta_{\rm N}$ corresponding to $\mu \in \{\mu_2, \mu_3\}$ and assume, for a contradiction, that $\psi$ has a closed nodal line $\cC$; let $\Omega_0 \subset \Omega$ be the open set for which $\partial \Omega_0 = \cC$. Then $\psi |_{\Omega_0}$ is an eigenfunction of the Laplacian on $\Omega_0$ with Dirichlet boundary conditions. Using domain monotonicity of Dirichlet Laplacian eigenvalues, we obtain, denoting by $\lambda_1 (\Omega_0)$ the smallest eigenvalue of the Dirichlet Laplacian on $\Omega_0$,
\begin{align*}
 \lambda_1 \leq \lambda_1 (\Omega_0) \leq \mu < \lambda_1,
\end{align*}
a contradiction.
\end{proof}

We point out once more that the absence of closed nodal lines for the eigenfunctions corresponding to $\mu_2$ has been known for long time, by the same argument as above and the inequality $\mu_2 < \lambda_1$ due to P\'olya \cite{P52}.

Before we proceed to the proof of Theorem \ref{thm:main}, we provide the following lemma, see, e.g., \cite[Lemma 2.2]{LR17}. Its proof is based on a unique continuation argument.

\begin{lemma}\label{lem:UCP}
Let $\Omega$ be a bounded, connected Lipschitz domain. Moreover, let $f \in H^1 (\Omega)$ be such that $- \Delta f = \lambda f$ holds in the sense of distributions, for some $\lambda \in \R$. If there exists a relatively open, non-empty set $\omega \subset \partial \Omega$ such that $f |_{\omega} = \partial_\nu f |_\omega = 0$ hold, then $f = 0$ identically in $\Omega$.
\end{lemma}

It should be pointed out that the lemma mentions the restriction of $\partial_\nu f |_{\partial \Omega} \in H^{-1/2} (\partial \Omega)$ to an open set $\omega \subset \partial \Omega$, for $f \in H^1 (\Omega)$ such that $\Delta f \in L^2 (\Omega)$. This can be made rigorously in the sense of distributions; cf.\ \cite{LR17}. However, in the present article we will apply the lemma only in situations where $f$ is regular enough so that we in fact take the restriction of a regular distribution, i.e.\ a function.

We can now prove Theorem \ref{thm:main}.

\begin{proof}[Proof of Theorem \ref{thm:main}]
As above, we denote by $\eta_1 \leq \eta_2 \leq \dots$ the positive eigenvalues of the operator $A$ defined in the previous section. Recall from Theorem \ref{thm:sameEV} that these eigenvalues coincide with the union of the positive numbers $\lambda_1, \lambda_2, \dots$ and $\mu_2, \mu_3, \dots$, including multiplicities. 

To prove the inequalities of the theorem, fix $k \in \N$ and let $\phi_j$, $j = 1, \dots, k$, be an orthonormal set of eigenfunctions of $- \Delta_{\rm D}$ such that $- \Delta_{\rm D} \phi_j = \lambda_j \phi_j$ holds for $j = 1, \dots, k$. Define
\begin{align}\label{eq:majestatischeAussicht}
 u = \binom{u_1}{u_2} = \sum_{j = 1}^k \left( \alpha_j \binom{\phi_j}{0} + \beta_j \binom{0}{\phi_j} \right),
\end{align}
where $\alpha_j, \beta_j \in \C$. Since the vector fields $\binom{\phi_j}{0}, \binom{0}{\phi_j}$, $j = 1, \dots, k$, form an orthonormal set in $L^2 (\Omega)^2$, they span a $2k$-dimensional subspace. Due to the Dirichlet boundary conditions of the $\phi_j$, each of the vector fields $u$ of the form \eqref{eq:majestatischeAussicht} belongs to $\dom \sa$, and 
\begin{align*}
 \int_\Omega |\nabla u_j|^2 \leq \lambda_k \int_\Omega |u_j|^2, \quad j = 1, 2.
\end{align*}
Thus
\begin{align}\label{eq:schonBesser}
\begin{split}
 \sa [u] & = \int_\Omega \left( |\partial_1 u_1|^2 + |\partial_2 u_2|^2 + |\partial_1 u_2|^2 + |\partial_2 u_1|^2 \right) \\
 & \quad + 2 \Real \int_\Omega \left( (\partial_1 u_1) \overline{(\partial_2 u_2)} - (\partial_1 u_2) \overline{(\partial_2 u_1)} \right) \\
 & = \int_\Omega |\nabla u_1|^2 + \int_\Omega |\nabla u_2|^2 \leq \lambda_k \int_\Omega |u|^2
\end{split}
\end{align}
holds for all $u$ of the form \eqref{eq:majestatischeAussicht}, where we have used integration by parts and the Dirichlet boundary conditions of $u_1$ and $u_2$.

Let now, in addition, $v \in \ker (A - \lambda_k)$ be arbitrary. Then for all $w \in \dom \sa$,
\begin{align*}
 \sa [v, w] = (A v, w) = \lambda_k (v, w)
\end{align*}
holds. In particular, together with \eqref{eq:schonBesser} we get
\begin{align}\label{eq:guteAbschatzung}
\begin{split}
 \sa [u + v] & = \sa [u] + 2 \Real \sa [v, u] + \sa [v] \leq \lambda_k \int_\Omega |u|^2 + 2 \lambda_k \Real \int_\Omega \langle v, u\rangle + \lambda_k \int_\Omega |v|^2 \\
 & = \lambda_k \int_\Omega |u + v|^2.
\end{split}
\end{align}

In order to conclude \eqref{eq:almostGoal} from \eqref{eq:guteAbschatzung}, we prove next that
\begin{align}\label{eq:stabil}
 \dim \left( \{u~\text{of the form~\eqref{eq:majestatischeAussicht}} \} + \ker (A - \lambda_k) \right) \geq 2 k + \dim \ker (- \Delta_{\rm D} - \lambda_k).
\end{align}
Indeed, the vector fields $u$ of the form \eqref{eq:majestatischeAussicht} span a $2k$-dimensional subspace of $\dom \sa$, and the fields $v = \nabla^\perp \phi$ with $\phi \in \ker (- \Delta_{\rm D} - \lambda_k)$ span a subspace of $\ker (A - \lambda_k)$ of dimension $\dim \ker (- \Delta_{\rm D} - \lambda_k)$. Moreover, assume there exists $u$ belonging to both spaces, i.e.\ $u$ can be written \eqref{eq:majestatischeAussicht} and $u = \nabla^\perp \phi$ for some $\phi \in \ker (- \Delta_{\rm D} - \lambda_k)$. Then $\phi$ belongs to $H^2 (\Omega)$ (since $\nabla^\perp \phi = u \in H^1 (\Omega)^2$), and
\begin{align*}
 \partial_\nu \phi |_{\partial \Omega} = \langle \nabla \phi |_{\partial \Omega}, \nu \rangle = - \langle u^\perp |_{\partial \Omega}, \nu \rangle = 0
\end{align*}
holds, as the components of $u$ have a vanishing trace on $\partial \Omega$. However, the latter identity together with $\phi |_{\partial \Omega} = 0$ constantly and $- \Delta \phi = \lambda_k \phi$ implies $\phi = 0$ in $\Omega$ according to Lemma \ref{lem:UCP}. Thus we have proven \eqref{eq:stabil}.

Now let us conclude \eqref{eq:almostGoal}. By combining \eqref{eq:guteAbschatzung} and \eqref{eq:stabil} we obtain
\begin{align*}
 \eta_{2 k + \dim \ker (- \Delta_{\rm D} - \lambda_k)} \leq \lambda_k;
\end{align*}
in other words, in total, $- \Delta_{\rm N}$ and $- \Delta_{\rm D}$ together have at least $2 k + \dim \ker (- \Delta_{\rm D} - \lambda_k)$ eigenvalues in $(0, \lambda_k]$, counted with multiplicities. As $- \Delta_{\rm D}$ has at most $k - 1 + \dim \ker (- \Delta_{\rm D} - \lambda_k)$ eigenvalues in $(0, \lambda_k]$, the number of eigenvalues of $- \Delta_{\rm N}$ in $(0, \lambda_k]$ must at least be
\begin{align*}
 2 k + \dim \ker (- \Delta_{\rm D} - \lambda_k) - \big(k - 1 + \dim \ker (- \Delta_{\rm D} - \lambda_k) \big) = k + 1. 
\end{align*}
Taking into account the eigenvalue $\mu_1 = 0$, this is equivalent to \eqref{eq:almostGoal}.

Let us now come to the sufficient conditions for strict inequality. The latter follows from \eqref{eq:guteAbschatzung} whenever we can show
\begin{align}\label{eq:dasIstES}
 \{u~\text{of the form~\eqref{eq:majestatischeAussicht}} \} \cap \ker (A - \lambda_k) = \{0\};
\end{align}
namely, in this case, the operator $A$ has at least $2 k + \dim \ker (A - \lambda_k)$ eigenvalues in $(0, \lambda_k]$, implying at least $2 k$ eigenvalues in $(0, \lambda_k)$. Out of these are at most $k - 1$ Dirichlet Laplacian eigenvalues, so that $- \Delta_{\rm N}$ has at least $k + 1$ eigenvalues in $(0, \lambda_k)$ and, hence, $k + 2$ eigenvalues in $[0, \lambda_k)$. This is equivalent to \eqref{eq:wasWillManMehr}.

Now we verify~\eqref{eq:dasIstES} under the conditions for strict inequality specified in the theorem. Consider first the case where $\lambda_k$ is a simple eigenvalue of $- \Delta_{\rm D}$. Assume that $u$ is of the form \eqref{eq:majestatischeAussicht} and, at the same time, belongs to $\ker (A - \lambda_k)$. Based on the differential equation $A u = \lambda_k u$ and the structure of the eigenspaces of $A$ analyzed in Proposition \ref{prop:spec} this can be written
\begin{align*}
 u = \binom{\alpha}{\beta} \phi_k = \nabla \psi + c \nabla^\perp \phi_k,
\end{align*}
where $\psi \in \ker (- \Delta_{\rm N} - \lambda_k)$, $\phi_k \in \ker (- \Delta_{\rm D} - \lambda_k)$ is non-trivial, and $\alpha, \beta, c$ are constants. Assume for a contradiction that $(\alpha, \beta) \neq (0, 0)$. We have
\begin{align}\label{eq:ersteHalfte}
 \left\langle \nabla \phi_k, \binom{\overline \alpha}{\overline \beta}^\perp \right\rangle = - \omega (u) = - c \Delta \phi_k = \lambda_k c \phi_k \in H_0^1 (\Omega)
\end{align}
and 
\begin{align*}
 \left\langle \nabla \phi_k, \binom{\overline \alpha}{\overline \beta} \right\rangle = \diver u = \Delta \psi = - \lambda_k \psi \in H^1 (\Omega).
\end{align*}
Thus, $\nabla \phi_k \in H^1 (\Omega)^2$ and, hence, $\phi_k \in H^2 (\Omega)$. Moreover, decomposition of the trace of $\nabla \phi_k$ into its normal and tangential components gives
\begin{align}\label{eq:naAlso}
 \left\langle \nu, \binom{\overline \alpha}{\overline \beta}^\perp \right\rangle \partial_\nu \phi_k |_{\partial \Omega} & = \left\langle \nabla \phi_k, \binom{\overline \alpha}{\overline \beta}^\perp \right\rangle \Big|_{\partial \Omega} + \left\langle \tau, \binom{\overline \alpha}{\overline \beta}^\perp \right\rangle \left\langle \nabla^\perp \phi_k |_{\partial \Omega}, \nu \right\rangle = 0,
\end{align}
where we have employed \eqref{eq:ersteHalfte} and Lemma \ref{lem:DTangential}. Note that there exists a relatively open, non-empty set $\Gamma \subset \partial \Omega$ such that $(\overline{\alpha}, \overline{\beta})^\top$ and $\nu (x)$ are linearly independent for almost all $x \in \Gamma$. Hence, \eqref{eq:naAlso} yields $\partial_\nu \phi_k |_\Gamma = 0$, and since $\phi_k |_{\partial \Omega} = 0$, Lemma \ref{lem:UCP} implies $\phi_k = 0$ constantly in $\Omega$, a contradiction. This gives \eqref{eq:dasIstES}.

Consider now the case that $\Omega$ is such that $\partial \Omega$ contains a line segment, i.e.\ an open set $\Sigma \subset \partial \Omega$, on which the exterior unit normal (and, thus, also the unit tangent vector) is constant. Let $T, N \in \R^2$ be unit vectors such that
\begin{align*}
 T = \tau \quad \text{and} \quad N = \nu \quad \text{on}~\Sigma;
\end{align*}
in particular, $(T, N)$ is an orthonormal basis of $\R^2$. If $u$ is of the form \eqref{eq:majestatischeAussicht} and belongs to $\ker (A - \lambda_k)$, then $u = \nabla \psi + \nabla^\perp \phi$ for some $\psi \in \ker (- \Delta_{\rm N} - \lambda_k)$ and $\phi \in \ker (- \Delta_{\rm D} - \lambda_k)$, and thus
\begin{align}\label{eq:nabitte}
 \omega (u) = \Delta \phi = - \lambda_k \phi \in H_0^1 (\Omega).
\end{align}
Note that there exists a non-empty, open set $\Sigma' \subset \Sigma$ such that $\overline{\Sigma'} \subset \Sigma$ on which $u$ is smooth up to the boundary; cf., e.g., \cite[Theorem 4.18 (ii)]{McL}. Moreover,
\begin{align*}
 \nabla u_j |_{\Sigma'} = \left\langle \nabla u_j |_{\Sigma'}, N \right\rangle N + \left\langle \nabla u_j |_{\Sigma'}, T \right\rangle T = N \partial_\nu u_j |_{\Sigma'}, \quad j = 1, 2,
\end{align*}
by Lemma \ref{lem:DTangential}. Therefore, \eqref{eq:nabitte} yields
\begin{align*}
 0 & = \omega (u) |_{\Sigma'} = \partial_1 u_2 |_{\Sigma'} - \partial_2 u_1 |_{\Sigma'} = N_1 \partial_\nu u_2 |_{\Sigma'} - N_2 \partial_\nu u_1 |_{\Sigma'} \\
 & = \left\langle \binom{\partial_\nu u_1 |_{\Sigma'}}{\partial_\nu u_2 |_{\Sigma'}}, N^\perp \right\rangle = \left\langle \binom{\partial_\nu u_1 |_{\Sigma'}}{\partial_\nu u_2 |_{\Sigma'}}, T \right\rangle = \partial_\nu \left( T_1 u_1 + T_2 u_2 \right) |_{\Sigma'}.
\end{align*}
As $T_1 u_1 + T_2 u_2 \in \ker (- \Delta_{\rm D} - \lambda_k)$, Lemma \ref{lem:UCP} implies 
\begin{align*}
 T_1 u_1 + T_2 u_2 = 0
\end{align*}
constantly in $\Omega$, i.e., $u_1$ and $u_2$ are linearly dependent, in other words,
\begin{align*}
 u = \binom{\alpha}{\beta} \phi_k
\end{align*}
for some non-trivial $\phi_k \in \ker (- \Delta_{\rm D} - \lambda_k)$ and constants $\alpha, \beta$. Now the exact same reasoning as above yields $u = 0$; thus, we have shown \eqref{eq:dasIstES} also in this case.
\end{proof}

\begin{remark}
Although Theorem \ref{thm:sameEV} provides a variational expression also in the case that $\Omega$ is multiply connected, it seems not obvious how our proof of Theorem \ref{thm:main} can be extended to this case. The test functions employed there will in general not be orthogonal to the space $H_{\rm c}$.
\end{remark}

We conclude this article by discussing optimality of the index shift in Theorem \ref{thm:main}.

\begin{example}\label{ex:disk}
Let $\Omega$ be the unit disk in $\R^2$. Then the eigenvalues of $- \Delta_{\rm D}$ are the squared zeroes of the Bessel functions $J_0$ (with multiplicity one) and $J_1, J_2, \dots$ (with multiplicity two). On the other hand, the positive eigenvalues of $- \Delta_{\rm N}$ are the squares of the positive roots of the derivatives $J_0'$ (with multiplicity one) and $J_1', J_2', \dots$ (with multiplicity two); see, e.g., \cite[Section 9.5.3]{PR}. Concretely,
\begin{align*}
 \mu_1 & = 0, \\
 \mu_2 = \mu_3 & \approx 1.84^2, \\
 \mu_4 = \mu_5 & \approx 3.05^2,
\end{align*}
while
\begin{align*}
 \lambda_1 & \approx 2.40^2, \\
 \lambda_2 = \lambda_3 & \approx 3.83^2.
\end{align*}
Hence $\mu_4 > \lambda_1$ but, for instance, $\mu_5 < \lambda_2$, and this remains true for small, sufficiently regular perturbations of $\Omega$, including such where no longer $\lambda_2 = \lambda_3$. This indicates that improved inequalities for selected, but not all, eigenvalues are possible.
\end{example}

\section*{Acknowledgements}
The author gratefully acknowledges financial support by the grant no.\ 2022-03342 of the Swedish Research Council (VR). 

\section*{Data availability statement}
Data sharing not applicable to this article as no datasets were generated or analysed during the current study.

\section*{Conflict of interest statement}
There is no conflict of interest.

\end{document}